\documentclass[11pt]{amsart}

\swapnumbers

\usepackage{amssymb,amsfonts}
\usepackage[mathscr]{eucal}
\usepackage[alphabetic]{amsrefs}
\usepackage{microtype}

\usepackage[ps,matrix,arrow,curve,cmtip]{xy}
\title{Modular isogeny complexes}
\author{Charles Rezk}
\date{ \today}
\address{Department of Mathematics \\
University of Illinois at Urbana-Champaign \\ 
Urbana, IL}
\email{rezk@math.uiuc.edu}
\urladdr{http://www.math.uiuc.edu/~rezk}


\numberwithin{equation}{section}

\makeatletter
  \let\c@subsection\c@equation
\makeatother

\theoremstyle{plain}   %% This is the default, anyway

\newtheorem{thm}[subsection]{Theorem}
\newtheorem{prop}[subsection]{Proposition}
\newtheorem{cor}[subsection]{Corollary}
\newtheorem{lemma}[subsection]{Lemma}

\theoremstyle{remark}
\newtheorem{rem}[subsection]{Remark}    
\newtheorem{exam}[subsection]{Example}

\theoremstyle{plain}

\DeclareMathOperator{\Cok}{Cok}

\DeclareMathOperator{\Ker}{Ker}

\newcommand{\Hom}{{\operatorname{Hom}}}

\newcommand{\ra}{\rightarrow}

\newcommand{\xra}{\xrightarrow}

\newcommand{\len}[1]{\lvert#1\rvert}
\newcommand{\set}[2]{{\{\,#1\mid#2\,\}}}

\newcommand{\powser}[1]{[\![#1]\!]}

\newcommand{\F}{\mathbb{F}}
\newcommand{\N}{\mathbb{N}}
\newcommand{\dfn}{\textbf}

\def\defeq{\overset{\mathrm{def}}=}

\begin{document}

\newcommand{\laurser}[1]{(\!(#1)\!)}

\newcommand{\M}{\mathcal{M}}
\newcommand{\bM}{\overline{\mathcal{M}}}
\renewcommand{\O}{\mathcal{O}}
\newcommand{\R}{\mathcal{R}}
\newcommand{\Z}{\mathbb{Z}}
\newcommand{\Q}{\mathbb{Q}}
\newcommand{\HH}{\mathbb{H}}
\newcommand{\Spec}{\operatorname{Spec}}
\newcommand{\charac}{\operatorname{char}}

\newcommand{\Schemes}{\mathrm{Schemes}}
\newcommand{\Sets}{\mathrm{Sets}}
\newcommand{\Rings}{\mathrm{Rings}}

\newcommand{\Isog}[1]{[{#1}\mathrm{-Isog}]}
\newcommand{\Ell}{(\mathrm{Ell})}
\newcommand{\Sch}[1]{(\mathrm{Sch}/#1)}
\newcommand{\Def}{\mathrm{Def}}

\newcommand{\Frob}{\operatorname{Frob}}
\newcommand{\univ}{\mathrm{univ}}

\newcommand{\nerve}{\operatorname{nerve}}

\newcommand{\mor}{\mathrm{mor}}
\newcommand{\rank}{\operatorname{rank}}

\newcommand{\disorder}{\operatorname{disorder}}
\newcommand{\cF}{\mathscr{F}}
\newcommand{\gr}{\mathrm{gr}}

\newcommand{\cS}{\mathscr{S}}
\newcommand{\cK}{\mathscr{K}}
\newcommand{\cX}{\mathscr{X}}

%%% abstract
\begin{abstract}
We describe a vanishing result on the cohomology of a cochain complex
associated to the moduli of chains of finite subgroup schemes on
elliptic curves.  These results have applications to algebraic
topology, in particular to the study of power operations for Morava
$E$-theory at height $2$. 
\end{abstract}

%%% the title
\maketitle

%%% table of contents
%\tableofcontents

\section{Introduction}

Let $A$ be an abelian group (possibly infinite), $k$ a commutative
ring, $p$ a prime, and $r\geq1$.  Let $\cK_{p^r}^\bullet(A;k)$ be the
cochain complex defined by 
\[
\cK_{p^r}^q(A;k) = \prod_{G_1,\dots,G_q} k,\qquad (\delta f)(G_1,\dots,G_{q+1})=\sum(-1)^k
f(G_1,\dots,\widehat{G_k},\dots, G_{q+1}),
\]
where the product is taken over the set of all increasing chains $G_1\subsetneq
\cdots \subsetneq G_q$ of subgroups of $A$ such that the largest
subgroup in the chain $G_q$ has order $p^r$.
One can show that 
\begin{enumerate}
\item $H^q\cK_{p^r}^\bullet(A;k)=0$ unless $q=r$,
\item $H^r\cK_{p^r}(A;k)$ is a free $k$-module, of rank $n\,p^{r(r-1)/2}$,
  where $n$ is the number of distinct subgroups of $A$ which are
  isomorphic to $(\Z/p)^r$.
\end{enumerate}
This is not a very deep result; it is essentially the theorem of
Solomon-Tits on the cohomology of the Tits building of $GL_r(\Z/p)$.
(This is described in \S\ref{sec:case-A}.)

Nonetheless,  we may consider an elliptic curve $E$ over
an algebraically closed field $k$, of characteristic $0$ or finite
characteristic other than $p$.  Taking $A=E(k)$ the group of
$k$-rational points on $E$, we see that since the $p$-torsion in $A$
is isomorphic to $(\Q_p/\Z_p)^2$, we must have $H^r\cK_{p^r}(E(k);k)=0$ for
$r\geq3$, while $H^1\cK_p(E(k);k)$ and $H^2\cK_p(E(k);k)$ are free
modules of ranks $p+1$ and $p$ respectively.

The goal of this paper is to prove an analogue of this (trivial)
calculation of $H^*\cK_{p^r}^\bullet(E(k);k)$, in which the fixed elliptic
curve over $k$ is replaced by an elliptic curve $E$ over an affine scheme
$S=\Spec A$, and in which the corresponding complex
$\cK^\bullet_{p^r}(E/S)$ is constructed not using the concrete
subgroups of the group of rational points of a fixed curve, but rather from the
moduli of finite subgroup schemes of $E/S$.  In particular, we obtain
a result which makes sense in characteristic $p$ (which for us is the
case of interest).  The motivation for proving such a result comes
from algebraic topology. When $E/S$ is a universal deformation of a
supersingular curve,   the complexes $\cK^\bullet_{p^r}(E/S)$ control
the homological properties of the algebra of power operations on the
version of elliptic cohomology associated to this curve; see
\S\ref{subsec:applications-alg-top} below.

\subsection{Subgroups of elliptic curves as a moduli problem}

Let $\Ell$ denote the category whose objects $E/S$ are elliptic schemes
$E$ over a base scheme $S$, and whose morphisms $E/S\ra E'/S'$ are
fiber squares.

Given an elliptic curve $E/S$, let $\Isog{N}(E/S)$ denote the set of
locally free finite commutative $S$-subgroup schemes $G\subset E$ which
are rank $N$ over $S$ \cite{katz-mazur}*{\S6.5}.  According to
\cite{katz-mazur}*{6.5.1}, $\Isog{N}$ is relatively representable and
finite over $\Ell$.  That is, given an elliptic curve $E/S$, the
functor on $\Sch{S}$ given by $T\mapsto \Isog{N}(E_T/T)$ is
represented by an $S$-scheme $\Isog{N}_{E/S}$
\cite{katz-mazur}*{\S4.2}, which is finite and flat, and hence locally
free, over $S$; 
furthermore, the rank of $\Isog{N}_{E/S}$ is \emph{constant}, and is
equal to the number of subgroups of order $N$ in $(\Q/\Z)^2$.

For each element $G\subset E$ of $\Isog{N}(E/S)$ there
is an  $N$-isogeny $f\colon E\ra E'$ of curves over $S$ with kernel
$G$, unique up to unique isomorphism in the category of isogenies with
domain $E$, hence the notation.

Now suppose that $S=\Spec A$ is an affine scheme.  Then
$\Isog{N}(E/S)$ is necessarily affine.  I write $\cS_N(E/S)$ for the function
ring of $\Isog{N}(E/S)$; it is naturally an $A$-algebra, finite and
locally free.  Furthermore,
given a map $T\ra S$ of schemes induced by a map $A\ra B$ of rings, we
have $\cS_N(E_T/T)\approx \cS_N(E/S)\otimes_A B$.

\begin{rem}
If we use the language of moduli stacks, then we can say that $\cS_N$
is a coherent sheaf on the moduli stack of elliptic curves; in fact,
it is the direct image of the structure sheaf along the evident map of stacks
$\mathscr{M}_{\Isog{N}}\ra \mathscr{M}_{\Ell}$ which forgets about the
subgroup.  Likewise, the complex
$\cK_N^\bullet$ to be defined  below is a complex of coherent sheaves on
$\mathscr{M}_{\Ell}$. 
We prefer to avoid
stack language, and discuss these objects in more concrete terms.  
\end{rem}

\subsection{Chains of subgroups}
\label{subsec:chains-of-subgroups}

For integers $N_1,\dots,N_q\geq1$, let
$\Isog{N_1,\dots,N_q}(E/S)$ denote
the set of sequences $\underline{G}=(G_1\subsetneq \dots \subsetneq G_q)$,
where $G_i$ is a locally 
free commutative $S$-subgroup scheme of $E$ of rank
$N_1\cdots N_i$, and where $G_{i-1}$ is a subscheme of $G_i$.  Thus,
$G_i/G_{i-1}$ is a finite group scheme over $S$ of rank $N_i$.
As above, given $E/S$, the functor on $\Sch{S}$ given by $T\mapsto
\Isog{N_1,\dots, N_q}(E_T/T)$ is represented by an $S$-scheme
$\Isog{N_1,\dots,N_q}_{E/S}$, finite and locally free over $S$.
Again, if $S=\Spec(A)$ is affine, so is $\Isog{N_1,\dots,N_q}_{E/S}$ with
function ring denoted
$\cS_{N_1,\dots,N_q}(E/S)$, and this ring is finite and locally free
as an $A$ module.  

As usual, elements $\underline{G}$ of $\Isog{N_1,\dots,N_q}({E/S})$ can
be identified with suitable isomorphism classes of sequences
\[
E\xra{f_1} E_1\xra{f_2} \dots \xra{f_q} E_q
\]
of isogenies with $\Ker(f_if_{i-1}\cdots f_i)=G_i$.

There are maps $U_i\colon \Isog{N_1,\dots,N_q}\ra \Isog{N_1,\dots
  ,N_{i-1}N_{i},\dots, N_q}$ for $i=1,\dots,q$, defined by
forgetting the $i$th group in the sequence $(G_1\subsetneq\cdots \subsetneq
G_q)$.  We write $u_i\colon \cS_{N_1,\dots,N_{i-1}N_i,\dots, 
  N_q}(E/S)\ra \cS_{N_1,\dots,N_q}(E/S)$ for the corresponding map of
rings, when $S$ is affine.

There are also evident maps $\Isog{N_1,\dots,N_q}\ra \Isog{1}$, which
forget about the information abut 
finite subgroups.  I'll write $s\colon \cS_1(E/S)\ra \cS_{N_1,\dots,N_q}(E/S)$ for
the corresponding map of rings, which is precisely the map which
exhibits $\cS_{N_1,\dots,N_q}(E/S)$ as an $A$-algebra. 

\subsection{The modular $N$-isogeny complex}

Fix an elliptic curve $E/S$, where $S=\Spec(A)$ is an affine scheme,
and let $N\geq 1$.
We define a bounded cochain complex $\cK^\bullet_N(E/S)$ of
$A$-modules as follows. 
follows.  Set 
\[
\cK^q_N(E/S) = \prod_{N_1,\dots,N_q} \cS_{N_1,\dots,N_q}(E/S),
\]
where the product runs through all tuples $(N_1,\dots,N_q)$ of
integers of length $q$ such that $N_1\cdots N_q=N$ and each $N_i>1$.  If
$q=0$ and $N>1$, we have $\cK^0_N=0$.  We also stipulate that if $N=1$, then
$\cK^0_1=\cS_1(E/S)=A$ and $\cK^q_1=0$ for $q>0$. 
Given an element $f\in \cK^q_N(E/S)$, write $f_{N_1,\dots,N_q}$ for
its component in $\cS_{N_1,\dots,N_q}(E/S)$.  
We define the coboundary map $\delta\colon \cK^{q-1}_N\ra \cK^{q}_N$ by
the formula
\[
(\delta f)_{N_1,\dots,N_q} = \sum_{i=1}^{q-1} (-1)^i u_i(f_{N_1,\dots,
  N_iN_{i+1},\dots, N_q}),
\]
where $u_i$ is as defined in \S\ref{subsec:chains-of-subgroups}.  

We will call $\cK^\bullet_N$ the \dfn{modular $N$-isogeny complex},
for lack of a better name.

We are mainly interested in the case when $N=p^r$ for some prime $p$.
For small values of $r$ these complexes appear as follows.
\[\xymatrix@R=1pt@C=35pt{
{\cK^\bullet_1\colon}
& {\cS_1}
\\
{\cK^\bullet_p\colon}
& {0} \ar[r] & {\cS_p} 
\\
{\cK^\bullet_{p^2}\colon}
& {0} \ar[r] & {\cS_{p^2}} \ar[r]^-{u_1} & {\cS_{p,p}}
\\
{\cK^\bullet_{p^3} \colon}
& {0} \ar[r] & {\cS_{p^3}} \ar[r]^-{(u_1,u_1)} & {\cS_{p^2,p}\times
  \cS_{p,p^2}} 
\ar[r]^-{(u_1,-u_2)} & {\cS_{p,p,p}}
}\]

\subsection{Main theorem}

Our main result is the following.
\begin{thm}\label{main-thm}
Let $E/S$ be an elliptic curve over an affine scheme $S=\Spec A$, and
let $p$ be a prime.
\begin{enumerate}
\item If $j\neq r$, then $H^j\cK_{p^r}^\bullet(E/S)=0$.
\item $H^r\cK_{p^r}^\bullet(E/S)$ is finite and locally free as an
  $A$-module.
\item $H^r\cK_{p^r}^\bullet(E/S)=0$ if $r\geq 3$.
\item There are natural isomorphisms of $A$-modules
\[
H^0\cK_{1}^\bullet(E/S)=\cS_1(E/S),\qquad
H^1\cK_p^\bullet(E/S)=\cS_p(E/S),
\]
\[
H^2\cK_{p^2}^\bullet(E/S)\approx \Cok[s\colon \cS_1(E/S)\ra \cS_p(E/S)].
\]
\end{enumerate}
\end{thm}
In particular, $H^0\cK_1^\bullet(E/S)$, $H^1\cK_p^\bullet(E/S)$, and
$H^2\cK_{p^2}(E/S)$ are locally free over $A$ of ranks $1$, $p+1$, and
$p$ respectively.

\subsection{Proofs of (3) and (4)}

The main claims of the theorem are (1) and (2), and we can deduce the
remaining statements from these.

\begin{proof}[Proof of (3) using (1) and (2)]
We use a ``dimension count'', meaning we compare ranks of finite and
locally free $A$-modules.  Write $\cS_{N_1,\dots,N_q}$ for
$\cS_{N_1,\dots,N_q}(E/S)$, and $\cK_{N_1,\dots,N_q}^\bullet$ for
$\cK_{N_1,\dots,N_q}^\bullet(E/S)$, with $S=\Spec A$.  
We have that $\cS_N$ is locally
free of \emph{constant} rank as an $A$-module, and that this rank is
equal to the number of subgroups of order $N$ in $(\Q/\Z)^2$.  Thus,
if we write $\sigma(N)=\sum_{d|N} d$ for the number of such subgroups,
then we have $\rank\cS_N =\sigma(N)$, and more generally, 
\[
\rank\cS_{N_1,\dots,N_r} = \sigma(N_1)\cdots \sigma(N_r).
\]
Counting such subgroups of $p$th power order leads to a
generating function
\[
f(T) = \sum_{r\geq0} \rank\cS_{p^r}\,T^r = \sum_{r\geq0}
\sigma(p^r)\, T^r= 
[(1-T)(1-pT)]^{-1},
\]
and from this we obtain
\[
\sum_{r\geq0} \rank\cK^q_{p^r}\,T^r = \sum_{r_1,\dots,r_q>0}
\rank \cS_{p^{r_1},\dots,p^{r_q}}\,T^{r_1+\cdots+r_q} =
\sum_{r_1,\dots,r_q>0} \sigma(r_1)\cdots
\sigma(r_q)T^{r_1+\cdots+r_q} = (f(T)-1)^q.
\]
By (2), the cohomologies $H^j\cK_{p^r}$ are finite and locally free
over $A$, so $\sum_j (-1)^j\rank H^j\cK_{p^r}= \sum_j(-1)^j\rank
\cK_{p^r}^j$.   
Hence by (1) we have
\[
\sum_{r\geq0}\rank H^r\cK_{p^r}^\bullet\, (-T)^r = \sum_{q\geq0}
(1-f(T))^q = (f(T))^{-1} =(1-T)(1-pT).
\]
Thus, $H^r\cK^\bullet_{p^r}=0$ if $r\geq 3$.
\end{proof}

\begin{proof}[Proof of (4) using (1) and (2)] 
The only nontrivial statement is that for $H^2\cK_{p^2}$.
Consider the following commutative square of moduli problems
\[\xymatrix{
{\Isog{p}} \ar[d]_{(E\xra{f}E')\mapsto (E\xra{f}E'\xra{\widehat{f}}E)} 
\ar[rrrr]^{(E\xra{f}E')\mapsto (E)}
&&&& {\Isog{1}} \ar[d]^{(E)\mapsto (E\xra{[p]}E)}
\\
{\Isog{p,p}} \ar[rrrr]_{(E\xra{f} E'\xra{g}E'')\mapsto (E\xra{gf}E'')}
&&&& {\Isog{p^2}} 
}\]
The commutativity of this square encodes the identity $\widehat{f}f=[p]$, where $\widehat{f}$
is the dual isogeny to a $p$-isogeny $f$.  
This square gives  a commutative square
\[\xymatrix{
{\cS_{p^2}(E/S)} \ar[r]^{u_1} \ar[d]
& {\cS_{p,p}(E/S)} \ar[d]
\\
{\cS_{1}(E/S)} \ar[r]_s
& {\cS_p(E/S)}
}\]
of $A$-algebras.
The map $s\colon \cS_1\ra \cS_p$ is
injective with 
locally free cokernel (it's faithfully flat), and the rank of the
cokernel is constant over $A$, with $\rank_A
\cS_p(E/S)/\cS_1(E/S) = (p+1)-1=p$.  Using statement (2)  of \eqref{main-thm},
$\cS_{p,p}(E/S)/\cS_{p^2}(E/S)\approx H_2\cK_{p^2}(E/S)$ is locally
free of 
rank $p$ as well.
The vertical maps in the square are epimorphisms since they admit
sections, so the induced map
$\cS_{p,p}/\cS_{p^2}\ra 
\cS_p/\cS_1$ is an epimorphism between locally free $A$-modules of
the same rank, hence is an isomorphism.
\end{proof}

\subsection{The proof of (1) and (2)}
\label{subsec:reduction-to-cases}

To complete the proof \eqref{main-thm}, note that if $E/S$ is an
elliptic curve, then $S$ admits a cover by open affines $U_i$ such
that $E_{U_i}/U_i$ is given by a Weierstrass equation, and so is the
pullback of an elliptic curve over a scheme of finite type.  
Thus, we may assume
without loss of generality that $S$ is of finite type, and in this
case  we
can replace ``locally free'' with ``projective'' in the statement of
the theorem.  

We use the following application of
Nakayama's lemma.
\begin{prop}\label{prop:nak-lemma}
Let $A$ be a commutative ring, and let $P^\bullet=(0\ra P^0\ra \cdots
\ra P^n\ra 0)$ be a bounded cochain complex of finitely generated projective
$A$-modules.  The following are equivalent.
\begin{enumerate}
\item $H^j(P^\bullet)=0$ for $j\neq n$, and $H^n(P^\bullet)$ is a
  finitely generated projective $A$-module.
\item For every ring homomorphism $A\ra B$, we have
  $H^j(P^\bullet\otimes_A 
  B)=0$ for   $j\neq n$, and $H^n(P^\bullet\otimes_A B)$ is a finitely
  generated projective $B$-module.
\item For every maximal ideal $\mathfrak{m}$ of $A$, we have
  $H^j(P^\bullet\otimes_A A/\mathfrak{m})=0$ for $j\neq n$.
\end{enumerate}
\end{prop}
 \begin{proof}
 Condition (1) implies that $P^\bullet$ is chain homotopy equivalent to
 the complex consisting of $H^n(P^\bullet)$ in degree $n$, and thus (2)
 follows.  Clearly (2) implies (3).
 
 To show that (3) implies (1), first consider the special case where
 $A$ is a local ring (in which case the $P^j$ are 
 finitely generated free modules), and write $k=A/\mathfrak{m}$.  This
 case is easily proved by
 induction on $n$, the case of $n=0$ being immediate.  If $n>0$, we
 have $H^0(P^\bullet\otimes_A k)=0$, so that
 $P^0\otimes_Ak\ra P^{1}\otimes_A k$ is injective.  Choose a free
 $A$-module $Q$ and a map $Q\ra P^{1}$ so that the induced map
 $g\colon P^{0}\oplus Q \ra P^{1}$ induces an isomorphism after tensoring
 down to $k$.  Thus $g$ is itself an isomorphism by Nakayama's lemma,
 and we conclude that 
 $P^\bullet$ is chain-homotopy-equivalent to the complex obtained from
 $P^\bullet$ by replacing $P^0$ with
 $0$ and $P^{1}$ with $Q$.  The induction hypothesis applies to $Q$,
 so the claim is proved.
 
 For general $A$, note that $H^*(P^\bullet)\otimes_A
 A_{\mathfrak{m}}\approx H^*(P^\bullet\otimes_AA_\mathfrak{m})$, and so
 from what we have shown it follows that $H^j(P^\bullet)=0$ for $j\neq
 n$.  The exact 
 sequence $P^{n-1}\ra P^n\ra H^n(P^\bullet)\ra0$ shows that
 $H^n(P^\bullet)$ is finitely presented, and therefore projective since it is
 locally free.
 \end{proof}

Thus, the proof of \eqref{main-thm} reduces to showing statements (1)
and (2) of \eqref{main-thm} for elliptic curves $E/S$, where $S=\Spec k$
for a field $k$. 
There are three cases we must consider, for a given prime $p$.
\begin{enumerate}
\item [(A)]  The field $k$ has characteristic not equal to $p$.
\item [(B)]  The field $k$ has characteristic $p$, and $E$ is an
  \emph{ordinary} elliptic curve.
\item [(C)]  The field $k$ has characteristic $p$, and $E$ is a
  \emph{supersingular} curve.
\end{enumerate}
In each case we may without loss of generality assume that $k$
algebraically closed.

There is a trick (inspired by its use in \cite{katz-mazur}*{Ch.\ 5})
which allows 
us to deduce case (B) from case (C). 
\begin{enumerate}
\item Note that for $E/S$, the question of whether the claims of
  \eqref{main-thm} are true for $E/S$ depends only on the underlying
  $p$-divisible group of $E$.  Over $S=\Spec(k)$ with $k$
  algebraically closed, there are only three $p$-divisible groups
  which can appear, according to the three cases: (A) $(\Q_p/\Z_p)^2$,
  (B) $\widehat{\mathbb{G}}_m\times   \Q_p/\Z_p$, and (C) the unique
  formal group of height $2$ over $k$.  

  In particular, to prove the theorem for case (B), it suffices to prove it for
  \emph{one} ordinary curve.

\item Given $E/S$ and $k\geq r$, let $U$ be the set of points $s$ in
  $S$ for which we have  that $H^j\cK^\bullet_{p^r}(E\otimes k(s)/\Spec
  k(s))=0$ for 
  $j\neq r$.  If $S$ is of finite type over $\Spec(\Z)$, then $U$ is a
  \emph{Zariski open} subset of $S$.   
\end{enumerate}
In the moduli stack of elliptic curves, every open neighborhood of a
supersingular point contains an ordinary point; more precisely, the
moduli stack admits an etale cover by a collection $\{E_i/S_i\}$ where
each $S_i$ is finite 
type over $\Spec(\Z)$, and such that each open neighborhood of a
supersingular point in $S_i$ contains an ordinary point.   Therefore
(C) implies 
(B).

We'll prove cases (A) and (C) by direct calculation.  Case (A) has an
easy combinatorial proof, while the proof of case (C) amounts to an explicit
calculation of the complex $\cK^\bullet_{p^r}$ at the universal
deformation of a supersingular curve, and will constitute the main
part of the paper.

Thus, in \S\ref{sec:case-A} we prove case (A) (\eqref{prop:koszul-when-p-inverted} and \eqref{prop:description-of-complex-when-p-inverted}).  In
\S\ref{sec:ss-formulas} we give an explicit description \eqref{prop:description-of-complex-ss} of the
complexes $\cK^\bullet_{p^r}(E/S)$, where $E/S$ is the universal
deformation of a suitable
supersingular curve; the main points we need are contained in
\cite{katz-mazur}*{Ch.\ 13}.  Finally, in \S\ref{sec:case-C} we use
this explicit description to prove
\eqref{main-thm} for the universal deformation (\eqref{prop:dual-formulation} and
\eqref{prop:koszul-for-ss}), from which case (C) follows directly;
this part of the  
argument is essentially an application of the proof of the ``PBW basis
theorem'' of \cite{priddy-koszul-resolutions}.

It is possible
to prove (B) by an explicit calculation such as that for (C) given here, and I
hope to give that calculation 
elsewhere. 

\subsection{Applications to algebraic topology}
\label{subsec:applications-alg-top}

This work was motivated by applications to elliptic cohomology (some
actual, some conjectural).  A detailed discussion is outside the scope
of this paper; however, we  can briefly describe a couple of  points of
contact. 

To every one-dimensional formal group $G_0$ of finite height over a
perfect field $k$ of characteristic $p$, there is an associated
generalized cohomology theory $E=E_{G_0/k}$, called the \dfn{Morava
  $E$-theory} associated to $G_0/k$.  The theory $E$ is $2$-periodic
and complex 
orientable, and its formal group $G/\pi_0E$ is the universal deformation of
$G_0/k$ in the sense of Lubin-Tate.  In particular, $\pi_0E\approx
\mathbb{W}k\powser{x_1,\dots,x_{n-1}}$ where $n$ is the height of
$G_0$.  

To each  cohomology theory $E$ is associated a certain  ring $\mathcal{P}$ of
operations, called the ring of
\dfn{power operations}; the ring $\mathcal{P}$ contains the
coefficient ring $\pi_0E$, but 
not centrally.  By work of Strickland
(\cite{strickland-finite-subgroups-of-formal-groups} and
\cite{strickland-morava-e-theory-of-symmetric}; see
\cite{rezk-congruence-condition} for an exposition), the ring $\mathcal{P}$
encodes in  a precise way all information about  the moduli of finite
subgroups of the formal 
group $G$.  

The reduction of $\mathcal{P}$ modulo $p$ makes a crucial appearance
in this paper.  More precisely, suppose that $G_0/\F_{p^2}$ is the
formal completion of a standard 
supersingular curve (defined in \S\ref{subsec:standard-ss}).  Its
universal deformation lives over the ring
$\mathbb{W}\F_{p^2}\powser{x}$; let 
$\mathcal{P}$ be the ring of power operations for the associated
Morava $E$-theory.  Then, as
a result of the calculations of \S\ref{sec:ss-formulas} and
\S\ref{sec:case-C}, we have an isomorphism of rings
\[
\mathcal{P}\otimes\Z/p \approx
\Gamma,
\]
where $\Gamma$ is the ring described in terms of explicit generators
and relations in \S\ref{sec:case-C}; see
\S\ref{subsec:relations-in-gamma} and
\S\ref{subsec:structure-of-gamma}, especially
\eqref{eq:gamma-formula-1}, \eqref{eq:gamma-formula-2}, and
\eqref{eq:gamma-formula-3}.  We note that the resulting admissible
monomial basis we give for $\Gamma$ is compatible with the the
calculations of 
Kashiwabara in \cite{kashiwabara-k2-homology}, though in his context
it is not possible to give an algebra structure. 

In the 1990s, Matt Ando, Mike Hopkins, and Neil Strickland conjectured
that for any Morava $E$-theory associated to any formal group, its
ring $\mathcal{P}$ of power operations would be what is called a
\dfn{Koszul ring}.  I have proved this conjecture (see
\cite{rezk-dyer-lashof-koszul}, forthcoming), using methods of
algebraic topology. 
The argument of \S\ref{sec:case-C} of this paper gives an independent
proof for the case of Morava $E$-theory  of height $2$ formal
groups, in some ways along the lines that Ando, Hopkins, and
Strickland envisioned.

The complexes $\cK^\bullet_{\ell^r}(E/S)$ when $\ell$ is
invertible over the scheme $S$ are closely related to 
work (see \cite{behrens-lawson-isogenies}) on approximations to
the $K(2)$-local      
sphere at a prime $p\neq\ell$, which are constructed using $\ell$th
power isogenies on a supersingular curve at $p$.

\subsection{Acknowledgments.} 

I'd like to thank Kevin Buzzard, who directed my attention to the
appropriate sections in \cite{katz-mazur}, and helpfully criticized my
fumbling 
formulation of an early version of some of this.   I would also like
to thank Nick Kuhn for supplying the observation that led to the
formulation of part (4) of the main theorem. 

The author was supported under NSF grant DMS--1006054.

\section{The proof of the theorem over fields with $p$ invertible}  
\label{sec:case-A}

Let $k$ be an algebraically closed field in which $p$ is invertible,
let $S=\Spec k$, and let $E/S$ be an elliptic curve.  In this section
we show that statements (1) and (2) of \eqref{main-thm} are true for
such a curve.  

In this case, we have that $E[p^\infty]\approx (\Q_p/\Z_p)^2$, and
thus that finite subgroup schemes $G\subset E$ of $p$th power order
correspond to finite subgroups of $(\Q_p/\Z_p)^2$.  The complex
$\cK_{p^r}(E/S)$ thus admits a combinatorial description, which we now
give.

\subsection{The order complex of subgroups of an abelian group}

Let $G$ be an abelian group, and let $P_G$ denote the poset of
proper non-trivial subgroups of $G$.  This poset is associated to an
abstract simplicial complex, called its \dfn{order complex}, which we
also denote $P_G$. This is an abstract simplicial complex whose vertices
correspond to proper non-trivial subgroups of $G$, and whose
$q$-simplices correspond to chains $[0\subsetneq G_1\subsetneq
\cdots\subsetneq 
G_q\subsetneq G]$ of subgroups $G_i$ of $G$.

Note that $P_{\Z/p}$ and $P_{0}$ are empty.

Given a simplicial complex $X$ with some chosen ordering of  its
vertices, let $C_\bullet(X)$ denote the usual chain complex associated
to $X$ with integer coefficients, and let  $\widetilde{C}_\bullet(X)$
denote the mapping fiber of 
the augmentation $C_\bullet(X)\ra \Z$.  Thus $\widetilde{C}_q(X)$ is free abelian
group on the $q$-simplices of $X$ if $q\geq0$, and $\widetilde{C}_{-1}(X)=\Z$.

\begin{prop}\label{prop:solomon-tits-generalization}
Let $G$ be a finite abelian $p$-group, with $G\neq0$.
\begin{enumerate}
\item If $pG\neq 0$, then $H_q(\widetilde{C}_\bullet(P_G))=0$ for all $q$.
\item If $pG=0$, so that $G\approx (\Z/p)^{\times r}$, then
  $H_q(\widetilde{C}_\bullet(P_G))=0$ for $q\neq r-2$, while
  $H_{r-2}\widetilde{C}_\bullet (P_G)$ is a free
  abelian group.
\end{enumerate}
\end{prop}
\begin{proof}
In case (1), 
there exists a proper subgroup $V\subsetneq G$ which is cyclic of
order $p$
and is \emph{not} a summand of $G$.  Thus, given any proper
non-trivial subgroup $H$ of $G$, the subgroup $H+V$ is again proper
and non-trivial.  The chain of inclusions $H\subseteq H+V \supseteq V$ defines a
pair of homotopies between self-maps of the geometric realization
$\len{P_G}$, which 
relate the identity map 
of $\len{P_G}$ to a constant map.  Thus, $\len{P_G}$ is contractible,
and the result on homology follows.

Case (2) is a special case of the theorem of Solomon-Tits
\cite{solomon-steinberg-character}, which says that $\len{P_G}$ is
homotopy equivalent to a wedge (one-point union) of
$(r-2)$-dimensional spheres if 
$r\geq2$.  An elegant proof which 
applies in this 
particular case is given in \cite{quillen-finite-generation}*{\S2}.
\end{proof}
In this paper, we actually only need part (2) of \eqref{prop:solomon-tits-generalization} in the cases
of $r=1$ and $2$, where it is trivial since $P_{(\Z/p)^r}$ is
empty or $0$-dimensional  in these cases.

\subsection{Description of $\cK^\bullet_{p^r}$}

Now we define for each finite abelian $p$-group $G$ and each abelian
group $M$ a cochain complex $D^\bullet_G(M)$ as follows.  If $G\approx
0$, we set $D^0_G(M)\approx M$ and $D^q_G(M)=0$ for $q\neq0$.  If
$G\not\approx 0$, we set
\[
D^q_G(M) = \Hom(\widetilde{C}_{q-2}(P_G), M),
\]
and the coboundary map of $D^\bullet_G(M)$ be induced by the boundary
map of $\widetilde{C}_{\bullet-2}(P_G)$.
We have the following immediate consequence of
\eqref{prop:solomon-tits-generalization}. 
\begin{prop}\label{prop:koszul-when-p-inverted}
Let $G$ be a finite abelian $p$-group.
\begin{enumerate}
\item If $pG\neq0$, then $H^q(D^\bullet_G(M))=0$ for all $q$.
\item If $pG=0$, so that $G\approx (\Z/p)^r$, then
  $H_q(D^\bullet_G(M))=0$ for $q\neq r$.
\end{enumerate}
\end{prop}

Now we consider our elliptic curve $E$ over an algebraically closed field $k$
in which $p$ is invertible.  
\begin{prop}\label{prop:description-of-complex-when-p-inverted}
Let $E/\Spec(k)$ be an elliptic curve, where $k$ is an algebraically
closed field not of characteristic $p$.  Then
$\cK_{p^r}^\bullet(E/S)\approx \prod_G D^\bullet_{G}(k)$ as cochain
complexes, where the product runs over all subgroups $G$ of
$E[p^\infty]\approx (\Q_p/\Z_p)^{\times 2}$ of order $p^r$.
\end{prop}
\begin{proof}
This is an explicit combinatorial identification, using the
isomorphism of rings $\cS_{p^{r_1},\dots,p^{r_q}}(E/S)\approx \prod
k$, where the product ranges over chains $0\subsetneq G_1\subsetneq
\cdots \subsetneq G_q\subseteq E[p^\infty]\approx (\Q_p/\Z_p)^2$ with
$\len{G_q}=p^r$.  
\end{proof}

\section{Supersingular curves, deformations, and isogenies}
\label{sec:ss-formulas}

As noted in \S\ref{subsec:reduction-to-cases}, the proof of the main
theorem reduces mainly to the special case of a supersingular curve $E_0/k$ over
a finite field.  We
will derive this  case from the case of  the universal deformation
$E_\univ/k\powser{x}$ of
$E_0/k$ to rings of characterstic $p$, which we handle by giving a
description of the complexes
$\cK^\bullet_{p^r}(E_\univ/k\powser{x})$ using explicit
formulas.  Although one could  consider deformations to rings
where $p$ is only assumed to be topological nilpotent, it is not
necessary to do so 
to prove the result we need; in any case I don't know how to construct 
explicit integral formulas.

In what follows, $E_0/k$ will be a fixed supersingular elliptic curve over a
perfect field $k$ of characteristic $p$.  

For any ring $R$ of characteristic $p$, we
write $\sigma=\sigma_R\colon R\ra R$ for the $p$th power ring
endomorphism $\sigma(r)=r^p$.  For an elliptic curve $E/\Spec R$, we
write $E^{(p^r)}=(\sigma^r)^*E$.  For an element $f(x)=\sum c_i x^i$
in a power series ring $R\powser{x}$, we will write
\[
f^{(p^r)} = \sum c_i^{p^r}x^i \in R\powser{x}.
\]

\subsection{The category of deformations}

Let $R$ be a local ring of characteristic $p$, and write
$k_R=R/\mathfrak{m}$.  
A \dfn{deformation} of 
$E_0$ to $R$ is data $(E,\psi,\alpha)$, where $E$ is an elliptic curve over
$\Spec R$, $\psi\colon k\ra k_R$ is a map of fields, and 
and $\alpha\colon E\otimes k_R\ra \psi^*E_0$ is an isomorphism of elliptic
curves over $\Spec k_R$. 

Let $(E_1,\psi_1,\alpha_1)$ and $(E_2,\psi_2, \alpha_2)$ be two
deformations of $E_0$ 
to $R$.   A \dfn{deformation of $F^r$} is an isogeny $f\colon E_1\ra
E_2$ of elliptic curves over $\Spec R$, such that
$\psi_2=\psi_1\circ \sigma^r$, and 
the square
\[\xymatrix{
{E_1\otimes k_R} \ar[r]^{f\otimes k_R} \ar[d]_{\alpha_1}
& {E_2\otimes k_R} \ar[d]^{\alpha_2}
\\
{\psi_1^*E_0} \ar[r]_{F^r}
& {\psi_2^*E_0}
}\]
commutes, where $F^r$ denotes the $p^r$-power relative Frobenius isogeny
$F^r\colon \psi_1^*E_0\ra \psi_1^*E_0^{(p^r)} = \psi_2^*E_0$.

A deformation of $F^r$ is necessarily a $p^r$-isogeny.
If $r=0$, we say that $f$ is an \dfn{isomorphism} between deformations.

The collection of all deformations of $E_0$ to $R$, and all
deformations of $F^r$ for $r\geq0$ between such, forms a category,
denoted $\Def(R)=\Def_{E_0/k}(R)$. 

\begin{exam}
Let $(E,\psi,\alpha)$ be a deformation of $E_0$ to $R$.  Then the
Frobenius isogeny $F^a\colon E\ra E^{(p^a)}$ is \emph{tautologically} a 
deformation of $F^a$; it gives a morphism $(E,\psi,\alpha)\ra
(E^{(p^a)},\psi\circ \sigma^a, \alpha^{(p^a)})$ in $\Def(R)$.
\end{exam}

Any isogeny between deformations of $E_0$ factors uniquely through some
deformation of $F^r$, and so any finite subgroup scheme of rank $p^r$ of a
deformation of $E_0$ is the kernel of an essentially unique
deformation of $F^r$.
\begin{prop}
Let $(E,\psi,\alpha)$ be a deformation of $E_0/k$ to $R$, and let
$G\subset E$ be a subgroup scheme finite and locally free over $\Spec R$ of
rank $p^r$.  Then there exists an isogeny $f\colon
(E,\psi,\alpha)\ra (E',\psi',\alpha')$ which is a deformation of $F^r$
and is such that $\Ker f=G$.
Given two such isogenies $f_i\colon (E,\psi,\alpha)\ra
(E_i',\psi_i',\alpha_i')$ for $i=1,2$, there exists a unique
isomorphism of deformations $g\colon (E_1,\psi_1,\alpha_1)\ra
(E_2,\psi_2,\alpha_2)$ such that $gf_1=f_2$.
\end{prop}
\begin{proof}
Given $G\subset E$, let $E'=E/G$ be the quotient curve, defined over
$\Spec R$.  Passing to $k$, we see that $G_0=G\otimes k$
is the unique subgroup scheme of rank $p^r$, and thus is the kernel of
$F^r$.  Thus 
there is a unique isomorphism
$\alpha'$ making the diagram
\[\xymatrix{
{E\otimes \bar{k}} \ar[r] \ar[d]_{\alpha}
& {(E/G)\otimes \bar{k}} \ar@{.>}[d]^{\alpha'}
\\
{\psi^*E_0} \ar[r]_{F^r}
& {{\psi'}^*E_0}
}\]
making the diagram commute, where $\psi'=\psi\circ \sigma^r$.

The second statement of the proposition is straightforward.
\end{proof}

There is at most one deformation of $F^r$ (for given $r$) between any
two deformations.
\begin{prop}\label{prop:def-of-frob-is-unique}
Let $R$ be an artinian local ring of characteristic $p$.  If
$f,f'\colon (E_1,\alpha_1,\psi_1)\ra (E_2,\alpha_2,\psi_2)$ are
deformations of $F^r$ in $\Def_{E_0/k}(R)$, then $f=f'$.
\end{prop}
\begin{proof}
Because $f$ and $f'$ are deformations of $F^r$, we have that $f\otimes
k_R=f'\otimes k_R\colon E_1\otimes k_R\ra E_2\otimes k_r$.  Thus
$(f-f')\otimes k_R$ is the $0$-homomorphism, whence $f-f'$ is the
$0$-homomorphism by ``rigidity'' \cite{katz-mazur}*{2.4.1}.
\end{proof}

\subsection{Universal deformation}

\begin{prop}
There is at most one isomorphism between any two deformations of $E_0$
to $R$.   There
is a universal deformation $E_\univ$ defined over $A\approx k\powser{x}$,
with the property that isomorphism classes of deformations of $E_0$ to
an artinian local ring
$R$ of characteristic $p$ are in bijective correspondence with local
homomorphisms of rings 
$A\ra R$.
\end{prop}
\begin{proof}
This is a standard result of deformation theory.  The Serre-Tate
theorem says that deformations of $E_0$ are the same as deformations
of its underlying formal group $\widehat{E}_0$, which is a formal
group of height $2$, and the deformations of such formal groups are
classified by a theorem of Lubin-Tate 
\end{proof}

Thus, the isomorphism class of a deformation $(E,\alpha,\psi)$ of
$E_0$ to $R$ corresponds,  
to a unique local ring homomorphism
$\phi_{(E,\alpha,\psi)}\colon 
A\ra R$.  If we make a choice of generator $x\in A$, so that $A\approx
k\powser{x}$, then we can speak of the \dfn{deformation parameter}
$x(E,\alpha,\psi)\defeq \phi_{(E,\alpha,\psi)}(x)\in \mathfrak{m}_R$.  
Thus,  deformations of $(E,\alpha,\psi)$ to an artinian local ring
$R$ correspond up to
isomorphism to pairs
$(\psi,a)$ consisting of a 
ring homomorphism $\psi\colon k\ra k_R$ and an element $a\in
\mathfrak{m}_R$, where $a=x(E,\alpha,\psi)$.

If two deformations are related by a Frobenius isogeny, their
deformation parameters are related in an obvious way.
\begin{prop}
Let $R$ be an artinian local ring of characteristic $p$.
Let $F^a\colon (E,\alpha,\psi)\ra
(E^{(p^a)},\alpha^{(p^a)},\psi\circ \sigma^a)$ denote the $p^a$-power
Frobenius viewed as a morphism between deformations of $E_0/k$ to
$R$.  Then we have that
\[
\phi_{(E^{(p^a)},\alpha^{(p^a)},\psi\circ \sigma^a)} =
\phi_{(E,\alpha,\psi)}\circ \sigma^a\colon A\ra R.
\]
\end{prop}
\begin{proof}
Immediate.
\end{proof}

\subsection{Standard supersingular curves over $\F_{p^2}$}
\label{subsec:standard-ss}

We'll say that a supersingular elliptic curve $E_0/k$ is
\dfn{standard} if 
$k=\F_{p^2}$ and $F^2=[-p]$.  Thus, to  prove case (C) of
\S\ref{subsec:reduction-to-cases}, it will suffice to prove it in the
case of \emph{standard} 
supersingular curves, by means of the following.

\begin{prop}
Every supersingular curve over a field containing $\F_{p^2}$ is
isomorphic to some standard curve 
$E_0/\F_{p^2}$.
\end{prop}
\begin{proof}
That all supersingular curves have models over $\F_{p^2}$ is well
known.  The statement about the $p^2$-power Frobenius is proved in
\cite{baker-gonzalez-poonen-finiteness-results}*{Lemma 3.21}.  See
also the discussion \cite{mathoverflow-ss-elliptic-curves}.
\end{proof} 

Given an elliptic curve $E$ over a ring $R$ of characteristic $p$, we
write $V^b=V^b_E\colon E^{(p^b)}\ra E$ for the \dfn{$p^b$-power
  Verschiebung 
  isogeny}, defined as the dual of the $p^b$-power Frobenius
$F^b_E\colon E\ra E^{(p^b)}$.  We write $(-V)^b=(-V_E)^b\colon
E^{(p^b)}\ra E$ 
for the composite $V^b_E\circ [(-1)^b]_{E^{(p^b)}}\colon E^{(p^b)}\ra
E$.  Our interest in 
standard supersingular curves comes from the following.
\begin{prop}\label{prop:verschiebung-is-def-frob}
Given a deformation $(E,\alpha,\psi)$ of a standard supersingular
curve $E_0/\F_{p^2}$ to $R$, the isogeny $(-V_E)^b$ is a deformation of
$F^b$.  That is, 
\[
(-V_E)^b\colon (E^{(p^b)}, \alpha^{(p^b)}, \psi\circ \sigma^b)\ra
(E,\alpha,\psi)
\]
is a morphism in $\Def(R)$.
In this case we have that
\[
\phi_{(E^{(p^b)},\alpha^{(p^b)},\psi\circ \sigma^b)} =
\phi_{(E,\alpha,\psi)}\circ \sigma^b\colon A\ra R.
\]
\end{prop}
\begin{proof}
Since $E_0$ is defined over $\F_{p^2}$, we have
$E_0^{(p^{2r})}=E_0$ for all $r$.
Since $E_0/\F_{p^2}$ is a standard supersingular curve, we have that
$F^2=-p=-VF$ 
on $E_0$, and therefore that 
\[
(-V_{E_0})^b=F_{E_0^{(p)}}^b\colon E_0^{(p^b)}\ra E_0^{(p^{2b})}=E_0. 
\]
Thus, the commutative diagram of elliptic curves and isogenies over $k_R$
\[\xymatrix{
{E^{(p^b)}\otimes k_R} \ar[r]^{(-V_E)^b} \ar[d]_{\alpha^{(p^b)}}
& {E\otimes k_R} \ar[d]^{\alpha}
\\
{\psi^*E_0^{(p^b)}} \ar[r]_{(-V)^b=F^b}
&{ \psi^*E_0}
}\]
shows that $(-V_E)^b$ is a deformation of $F^b$.
\end{proof}

\subsection{Isogenies of type $(a,b)$}

Let $S$ be an $\F_p$-scheme, and let $E_1$ and $E_2$ be elliptic
curves over $S$.  An \dfn{isogeny of type $(a,b)$}
\cite{katz-mazur}*{13.3.4} is a $p^{a+b}$-isogeny  $f\colon E_1\ra
E_2$ of curves over $S$ which admits a
factorization of the form 
\[
E_1 \xra{F^a} E_1^{(p^a)} \xra[\sim]{g} E_2^{(p^b)} \xra{V^b} E_2,
\]
where $g$ is an isomorphism.  
Equivalently, $f$ is of type $(a,b)$ if it admits a factorization of
the form
\[
E_1 \xra{F^a} E_1^{(p^a)} \xra[\sim]{g'} E_2^{(p^b)} \xra{(-V)^b} E_2.
\]

\begin{prop}
Let $E_0/k$ be a standard supersingular elliptic curve, and let
$(E_1,\alpha_1,\psi_1)$ and $(E_2,\alpha_2,\psi_2)$ be two
deformations of $E_0$ to an artinian local $\F_p$-algebra $R$.
Suppose $r=a+b$. 
The following are equivalent.
\begin{enumerate}
\item There exists a (necessarily unique)  isogeny $f\colon E_1\ra
  E_2$ which is (i) a 
  deformation of $F^r$, and (ii) of type $(a,b)$.
\item $\phi_{(E_1,\alpha_1,\psi_1)}\circ \sigma^a =
\phi_{(E_2,\alpha_2,\psi_2)}\circ \sigma^b$ as maps $A\ra R$.
\end{enumerate}
\end{prop}
\begin{proof}
Suppose $\phi_{(E_1,\alpha_1,\psi_1)}\circ \sigma^a=
\phi_{(E_2,\alpha_2,\psi_2)}\circ \sigma^b$, which means that there
exists an 
isomorphism $g\colon (E^{(p^a)},\alpha_1^{(p^a)}, \psi_1)\ra
(E^{(p^b)},\alpha_2^{(p^b)},\psi_2)$ in $\Def(R)$.  Then
$(-V)^b\circ g\circ F^a\colon E_1\ra E_2$ is a deformation of $F^r$
(using \eqref{prop:verschiebung-is-def-frob}) and an isogeny of type
$(a,b)$.

Conversely, consider a deformation of $F^r$ of the form $f=(-V)^b\circ
g\circ F^a\colon E_1\ra E_2$.  In the diagram
\[\xymatrix{
{E_1\otimes k_R} \ar[r]^{F^a} \ar[d]_{\alpha_1}
& {E_1^{(p^a)}\otimes k_R} \ar[r]^{g\otimes k_R} \ar[d]_{\alpha_1^{(p^a)}}
& {E_2^{(p^b)}\otimes k_R} \ar[r]^{(-V)^b}
\ar[d]^{\alpha_2^{(p^b)}}
& {E_2\otimes k_R} \ar[d]^{\alpha_2}
\\
{\psi_1^*E_1} \ar[r]_{F^a}
& {\psi_1^*E_1^{(p^a)}} \ar@{=}[r]
& {\psi_1^*E_1^{(p^a)}} \ar[r]_{F^b}
& {\psi_1^*E_1^{(p^{a+b})}}
}\]
the left-hand and right-hand squares, as well as the large rectangle,
commute.  Therefore we must have that
$\alpha_1^{(p^a)}=\alpha_2^{(p^b)}\circ (g\otimes k_R)$, whence
$(E_1^{(p^a)}, \alpha_1^{(p^a)},\psi_1\circ \sigma^a)$ and
$(E_2^{(p^b)}, \alpha_2^{(p^b)}, \psi_2\circ \sigma^b)$ are isomorphic
deformations.
\end{proof}

Given a choice of generator $x\in A$, we can restate this as follows.
\begin{cor}
Let $E_0/k$ be a standard supersingular curve.
Let $(E_1,\alpha_1,\psi_1)$ and $(E_2,\alpha_2,\psi_2)$ be two objects
of $\Def_{E_0/k}(R)$, with deformation parameters
$x_i=x(E_i,\alpha_i,\psi_i)\in R$ for $i=1,2$. 
There exists a (necessarily unique) morphism $f\colon
(E_1,\alpha_1,\psi_1)\ra (E_2,\alpha_2,\psi_2)$ in $\Def_{E_0/k}(R)$
of type $(a,b)$ if and only if
\begin{enumerate}
\item [(i)] $\psi_2=\psi_1\circ \sigma^{a+b}$, and 
\item [(ii)] $x_1^{p^a}=x_2^{p^b}$.
\end{enumerate}
\end{cor}
\begin{proof}
The only thing to note is that since $k=\F_{p^2}$,
(i) is equivalent to $\psi_1\circ \sigma^a=\psi_2\circ \sigma^b$.
\end{proof}

\subsection{Explicit description of the deformation category of a 
  standard supersingular curve}
\label{subsec:explicit-description-ss}

As before, $E_0/k$ is a standard supersingular curve.

Let $F_{p^r}(x,y)\in k[x,y]$ be the polynomial given by
\[
F_{p^r}(x,y)= \prod_{i+j=r}(x^{p^i}-y^{p^j}).
\]
\begin{prop}\label{prop:explicit-description-of-def-category}
Let $(E_1,\alpha_1,\psi_1)$ and $(E_2,\alpha_2,\psi_2)$ be two
deformations of $E_0$ to an artinian local $\F_p$-algebra $R$, with
deformation parameters 
$x_i=x(E_i,\alpha_i,\psi_i)\in R$.  There exists a (necessarily
unique) deformation of $F^r$
from
$(E_1,\alpha_1,\psi_1)$ to $(E_2,\alpha_2,\psi_2)$ if and only if
\begin{enumerate}
\item [(i)] $\psi_2=\psi_1\circ \sigma^r$, and
\item [(ii)] $F_{p^r}(x_1,x_2)=0$.
\end{enumerate}

Furthermore, there is a universal example of a deformation of $F^r$,
given by $f\colon s^*E_\univ\ra t^*E_\univ$, defined over the ring
$A_r\approx k\powser{x_1,x_2}/(F_{p^r}(x_1,x_2))$ with $s,t\colon
A\approx k\powser{x}\ra 
A_r$ given by $s(f(x))=f(x_1)$ and $t(f(x))=f^{(p^r)}(x_2)$.  
\end{prop}

\begin{rem}
Consider a different choice $x'\in k\powser{x}$ of deformation parameter, so
that $x'=f(x)=c_1x+\cdots \in k\powser{x}$ with $c_1\neq0$, and let
$x_1'=f(x_1)$ and $x_2'=f^{(p^r)}(x_2)$ in $k\powser{x_1,x_2}$.
Because $k=\F_{p^2}$, we have that $f^{(p^{2r-i})}(x)=f^{(p^i)}(x)$,
and thus 
\[
x_1'^{p^i}-x_2'^{p^{r-i}}= f(x_1)^{p^i} -
f^{(p^r)}(x_2)^{p^{r-i}} = f^{(p^i)}(x_1^{p^i}) -
f^{(p^i)}(x_2^{p^{r-i}}) = (x_1^{p^i}-x_2^{p^{r-i}})(\text{unit in
  $k\powser{x_1,x_2}$}). 
\]
Thus, the elements $F_{p^r}(x_1,x_2)$ and
$F_{p^r}(x_1',x_2')$ generate the same ideal in 
$k\powser{x_1,x_2}$; our description of $A_r$ does not depend on the
choice of deformation parameter.
\end{rem}

\begin{proof}
We have already noted that giving a deformation of $F^r$ with given
domain $(E_1,\alpha_1,\psi_1)$ is the same as giving a subgroup scheme
of order $p^r$.  
According to the discussion in \cite{katz-mazur}*{\S6.8 (especially
pp.\ 181--3)}, the universal example of such a subgroup scheme $G$ of a
deformation $E$ of $E_0$ is defined over  a ring of the form
$A_r=k\powser{x_1,x_2}/J$, where $J$ is a principal ideal; $x_1$ and
$x_2$ are the deformation 
parameters of $E$ and $E/G$ respectively.  Thus, it suffices to
describe a generator $g$ of $J$.  That we can take
$g=F_{p^r}(x_1,x_2)$ is the 
essential content of \cite{katz-mazur}*{13.4.6}, which is an
application of the ``crossings theorem'' \cite{katz-mazur}*{13.1.3}.

We can give a quick and dirty proof that  $J=(F_{p^r}(x_1,x_2))$.
As noted, we can write $J=(g)$ for some element $g$.
For $a+b=r$ the projection map
\[
\gamma_{ab}\colon A_r=k\powser{x_1,x_2}/(g) \ra 
k\powser{x_1,x_2}/(x_1^{p^a}-x_2^{p^b}),
\]
is precisely the ring homomorphism which classifies the universal
deformation of $F^r$ of type $(a,b)$.  For each $a+b=r$ write 
$g_{ab}=x_1^{p^a}-x_2^{p^b}$; the existence of $\gamma_{ab}$ shows that
$g_{ab}$ divides $g$.  We have that $g_{ab} =
f_{ab}^{p^{\min(a,b)}}$, where 
$f_{ab}$ is an irreducible element of $k\powser{x,y}$, and any pair of
the $f_{ab}$'s are distinct-up-to-units. 
Thus, since $k\powser{x_1,x_2}$ is a UFD,
the product $F_{p^r}(x_1,x_2)=\prod g_{a,b}$ must also divide
$g$.  
We know that $A_r$
(since it classifies subgroup schemes of order $p^r$) is finite and
free over $k\powser{x_1}$ of rank $1+p+\cdots+p^r$; thus 
\[
g\equiv x_2^{1+p+\dots+p^r}\cdot\text{(unit)} \equiv
F_{p^r}(x_1,x_2)\cdot 
\text{(unit)} \mod (x_1),
\]
and so $g=F_{p^r}(x_1,x_2)\cdot \text{(unit)}$ by Weierstrass preparation.
\end{proof}

As a result of the above proposition, the category $\Def(R)$ of
deformations of a 
standard supersingular curve to an artinian local $\F_p$-algebra is
equivalent to the category in which
\begin{enumerate}
\item [(1)] objects are pairs $(\psi,a)$ consisting of ring
  homomorphisms $\psi\colon k\ra k_R$
  and elements $a\in \mathfrak{m}_R$, and
\item [(2)] morphisms $(\psi_1,a_1)\to (\psi_2,a_2)$ are
  integers $r\geq0$ such that $\psi_2=\psi_1\circ \sigma^{r}$ and
  $F_{p^r}(a_1,a_2)=\prod_{i+j=r}(a_1^{p^i}-a_2^{p^j})=0$ in $R$.  
\end{enumerate}
It is not \emph{a priori} obvious that composition in the above
category well-defined (though it must be by
\eqref{prop:explicit-description-of-def-category}), 
since this would amount to showing that $F_{p^r}(a,b)=0$ and
$F_{p^{r'}}(b,c)=0$ imply $F_{p^{r+r'}}(a,c)=0$ for $a,b,c\in
\mathfrak{m}_R$.  
In the Appendix we give a direct proof of this fact about these
polynomials.

\subsection{Explicit description of $\cK_{p^r}^\bullet$ for universal
  deformations of a  supersingular curve}
\label{subsec:explicit-description-complex-ss}

Fix a universal deformation $E/S$ of a standard supersingular curve
$E_0/k$, where $S=\Spec A$.  Fix an isomorphism $A=k\powser{x}$, 
and write 
$A_r= \cS_{p^r}(E/S)$, and more generally
$A_{r_1,\dots,r_q}=\cS_{p^{r_1},\dots, p^{r_q}}(E/S)$.  

The discussion of \S\ref{subsec:explicit-description-ss} can be
summarized as follows. 
\begin{prop}\label{prop:description-of-complex-ss}
Let $s,t\colon A\ra A_r$ be the maps classifying respectively the
source and target of the universal deformation of $F^r$, as in
\eqref{prop:explicit-description-of-def-category}. 
\begin{enumerate}
\item There are isomorphisms
\[
A_r\approx k\powser{x_0,x_1}/(F_{p^r}(x_0,x_1)),
\]
such that $s\colon A\ra A_r$ is given by $s(f(x))=f(x_0)$ and $t\colon
A\ra A_r$ is given by $t(f(x))=f^{(p^r)}(x_1)$. 

\item 
There are isomorphisms
\begin{align*}
  A_{r_1,\dots,r_q} &\approx A_{r_1}{}^t\!\otimes_A\!{}^s\cdots
  {}^t\!\otimes_A \!{}^s A_{r_q}
\\
&\approx k\powser{x_0,\dots,x_q}(F_{p^{r_1}}(x_0,x_1),\dots,
F_{p^{r_q}}(x_{q-1},x_q)),
\end{align*}
where the map $s_k\colon A\ra A_{r_1,\dots,r_q}$ given by
$s_k(f(x))=f^{(p^{r_1+\cdots+r_k})}(x_k)$ classifies the quotient
curve $E/G_k$ (in the notation of
\S\ref{subsec:chains-of-subgroups}).

\item 
With respect to the above isomorphism, the map $u_k\colon
A_{r_1,\dots,r_{k-1}+r_k,\dots,r_q}\ra 
A_{r_1,\dots, r_q}$ is given by $u_k(x_i)=x_i$ if $i<k$, and
$u_k(x_i)=x_{i+1}$ if $i\geq k$.
\end{enumerate}
\end{prop}

This determines explicitly the structure of the complex
$\cK^\bullet_{p^r}(E/S)$.  
Thus, to prove case (C) of \S\ref{subsec:reduction-to-cases}, it suffices to prove (1) and (2) of
\eqref{main-thm} for this 
explicit complex.  

\subsection{Two useful lemmas}

The following two lemmas will be needed in the next section.
\begin{lemma}\label{lemma:socle-lemma}
For all $r\geq 1$, the homomorphism $u_1\colon A_{{r+1}}\ra
A_{1,{r}}$ is the inclusion of an $A$-module summand, where we regard
$A_{{r+1}}$ as an $A$-module by $s\colon A\ra A_{{r+1}}$.  
\end{lemma}
\begin{proof}
By Nakayama's lemma it is enough to prove that the map $u_1$ is
injective after tensoring down 
along $A\approx k\powser{x}\ra k$.   Thus, it suffices to show that
the ring homomorphism
\[
 k\powser{z}/(z^{1+p+\cdots+p^{r+1}}) \ra
k\powser{y,z}/(y^{1+p},F_{p^{r}}(y,z))=B
\]
sending $z\mapsto z$ 
is injective.  It will suffice to show that
$z^{p+\cdots+p^{r+1}}\neq 0$ in $B$.  Observe that $B$ has a basis
over $k$ given by the monomials $y^iz^j$ with $0\leq i\leq p$ and
$0\leq j\leq p+p^2+\cdots +p^r$.

In the target ring $B$ we have
\begin{align*}
0=F_{p^{r}}(y,z)^p &=
(y-z^{p^{r}})^p(y^p-z^{p^{r-1}})^p(y^{p^2}-z^{p^{r-2}})^p\cdots
(y^{p^{r}}-z)^p
\\
&= \pm (y^p-z^{p^{r+1}})z^{p^r}z^{p^{r-1}}\cdots z^p
\end{align*}
and thus $z^{p+p^2+\cdots+p^{r+1}}= y^pz^{p+p^2+\cdots+p^r}$ is one of
the elements of our $k$-basis for $B$, and thus is non-zero.
\end{proof}

\begin{lemma}\label{lemma:relations}
There is a split short exact sequence of $A$-modules
\[
0\ra A_2\xra{u_1} A_{1,1} \xra{\bar{v}} A_1/s(A)\ra 0,
\]
where $\bar{v}$ is the composition of the projection $A_1\ra A_1/s(A)$ with
the ring homomorphism $v\colon A_{1,1}\ra A_1$ defined by
\[
v(x_0)=x_0=v(x_2),\quad v(x_1)=x_1, 
\]
using the identifications
$A_{1,1}=k\powser{x_0,x_1,x_2}/(F_p(x_0,x_1),F_p(x_1,x_2))$ and
$A_1=k\powser{x_0,x_1}/(F_p(x_0,x_1))$ of
\eqref{prop:explicit-description-of-def-category}. 
\end{lemma}
\begin{proof}
Consider the commutative square of ring maps
\[\xymatrix{
{k\powser{x_0,x_1}/(F_{p^2}(x_0,x_1))} \ar[r]^-{u_1} \ar[d]_{w}
& {k\powser{x_0,x_1,x_2}/(F_p(x_0,x_1),F_p(x_1,x_2))} \ar[d]^{v}
\\
{k\powser{x}} \ar[r]_-{s} 
& {k\powser{x_0,x_1}/(F_p(x_0,x_1))}
}\]
where $w(x_0)=x=w(x_1)$.  The horizontal maps are injective, and the
vertical maps are surjective, and the induced map of cokernels $\Cok
u_1\ra \Cok s$ is a surjective map between free
$A=k\powser{x_0}$-modules of rank $p$, and thus is an isomorphism.
\end{proof}

\section{The proof of the theorem for supersingular curves}
\label{sec:case-C}

Recall that we wish to prove statements (1) and (2) of
\eqref{main-thm} for a standard
supersingular curve $E_0/k$.  It is clear that it
suffices to prove these statements for the universal deformation
$E/S$, where $S=\Spec A$ with $A=k\powser{x}$.  Thus, from now on we
fix such a 
universal deformation $E/S$. 

By \S\ref{subsec:explicit-description-complex-ss}, we have obtained an
explicit description of 
the complexes $\cK_{p^r}^\bullet(E/S)$.  Thus, we will prove the
desired results  by means of an explicit
calculation. 

\subsection{A dual formulation}

We will not work directly with the complex $\cK_{p^r}^\bullet(E/S)$,
but rather with its dual.
Let $K_\bullet^r$ be the chain complex which is $A$-linear dual to
$\cK^\bullet_{p^r}(E/S)$.   Thus $K_q^r = \Hom_A(\cK^q_{p^r}(E/S),A)$,
where the $A$-module structure on $\cK^q_{p^r}(E/S)$ (and thus on $K_q^r$)
is induced by the ring homomorphisms $s\colon A\ra
A_{r_1,\dots,r_q}$. 
To prove that $H^j(\cK^\bullet_{p^r}(E/S))=0$ for $j\neq r$, and is a
projective $A$-module for $j=r$, we will use the following observation.

\begin{prop}\label{prop:dual-formulation}
Let $A$ be a commutative ring and let
$P^\bullet=(0\ra P^0\ra\cdots \ra P^n\ra0)$ be  a bounded cochain
complex of finitely generated 
projective $A$-modules.  Let $P_\bullet=\Hom_A(P^\bullet,A)$ denote
the chain complex obtained by taking $A$-linear duals.
Then the following are equivalent.
\begin{enumerate}
\item $H^j(P^\bullet)=0$ for $j\neq n$, and $H^n(P^\bullet)$ is a
  finitely generated projective $A$-module.
\item $H_j(P_\bullet)=0$ for $j\neq n$.
\end{enumerate}
\end{prop}
\begin{proof}
Condition (1) says that the sequence $0\ra P^0\ra \cdots \ra P^n\ra
H^n(P^\bullet)\ra 0$ is an exact sequence of finitely generated
projective modules. 
Condition (2) says that the sequence $0\ra H_n(P_\bullet)\ra P_n\ra
\cdots \ra P_0\ra 0$ is a exact sequence of modules, all of which are
finitely generated projective except perhaps for $H_n(P_\bullet)$; but
then it is 
straightforward to show that $H_n(P_\bullet)$ must be projective and
finitely generated  as well.  The result then follows from the fact
that $\Hom_A({-},A)$ is a contravariant autoequivalence of the
category of finitely generated projectives.
\end{proof}

Thus, it will suffice to prove that $H_jK_\bullet^r=0$ for $j\neq r$.
The remainder of the section is devoted to the proof of this
\eqref{prop:koszul-for-ss}.

\subsection{Bimodules and duals}
We establish some notation.  Fix a commutative ring $A$.  Given
an $A$-bimodule $M$, we write $M^*$ for the set $\Hom_A(M,A)$ of left
$A$-module homomorphisms.  Then $M^*$ admits the structure of an
$A$-bimodule, defined by
\[
(a\cdot \phi\cdot b)(m)=\phi(a\cdot m\cdot b)
\]
for $a,b\in A$, $m\in M$, $\phi\in M^*$.

Given $A$-bimodules $M$ and $N$, we define a map
\[
M^*\otimes_A N^* \ra (M\otimes_A N)^*
\]
of $A$-bimodules by 
\[
(\phi\otimes\psi)(m\otimes n) = \phi(m\cdot \psi(n)).
\]
(One must check that this is well-defined, and actually gives a map of
bimodules.  Note that $A$ is commutative; the formulas we use here
don't make sense for non-commutative $A$.)  

More generally, we obtain $A$-bimodule maps
\[
M_1^*\otimes_A\cdots\otimes_A M_q^* \ra (M_1\otimes_A \cdots \otimes_A
M_q)^* 
\]
by
\[
(\phi_1\otimes\cdots\otimes \phi_q)(m_1\otimes \cdots m_q) = \phi_1(
m_1\cdot \phi_2( m_2\cdot \cdots \phi_q(m_q))).
\]
We note that if each $M_i$ is finitely generated and free as a left
$A$-module, then so is $M_1\otimes_A \cdots \otimes_A M_q$, and the
above map is an isomorphism.   

\subsection{The ring $\Gamma$}

We will describe the dual complex $K^r_\bullet$ 
in terms of a certain graded associative ring
$\Gamma=\bigoplus_{r\geq0} \Gamma_r$.  This ring will contain $A=k\powser{x}$ (in
fact, $\Gamma_0=A$), but $A$ will \textbf{not} be central in
$\Gamma$.

We will regard each ring $A_r=\cS_{p^r}(E/S)$ as an $A$-bimodule, with
the left 
$A$-module structure coming from $s\colon A\ra A_r$, and the right
$A$-module structure coming from $t\colon A\ra A_r$.  With this
notation we have an isomorphism of $A$-bimodules
$A_{r_1,\dots,r_q}= A_{r_1}\otimes_A \cdots \otimes_A A_{r_q}$, and
each of the maps $u_i\colon A_{r_1,\dots,
  r_{i-1}+r_i,\dots,r_q}\ra A_{r_1,\dots,r_q}$ is thus an $A$-bimodule
homomorphism. 

Let $\Gamma_r=A_r^*$.  As we have observed, $\Gamma_r$ is naturally an
$A$-bimodule.  In terms of explicit power series, the bimodule
structure is defined by
\[
(f(x)\cdot \phi\cdot g(x))(h(x_0,x_1))=
\phi(f(x_0)h(x_0,x_1)g^{(p^r)}(x_1)).
\]
Observe that $\Gamma_0=A^*\approx A$, and that since $A_r$ is finitely
generated and free as a left $A$-module, so is $\Gamma_r$.

From the above remarks, we see that the $A$-bimodule isomorphism
$A_r\otimes_A A_{r'}\approx 
A_{r,r'}$ gives an isomorphism $\Gamma_r\otimes_A
\Gamma_{r'}\ra A_{r,r'}^*$.  

Define a product $\mu\colon  \Gamma_r\otimes_A \Gamma_{r'}\ra
\Gamma_{r+r'}$ by
\[
(\mu(\phi\otimes \psi))(g)= (\phi\otimes \psi)(u_1(g)),
\]
where $\phi\in \Gamma_r$, $\psi\in \Gamma_{r'}$, and $g\in A_{r+r'}$.
That is, $\mu$ is dual to the $A$-bimodule map $u_1\colon A_{r+r'}\ra
A_{r,r'}$. 
This makes $\Gamma=\bigoplus_r \Gamma_r$ into a graded
associative ring, which contains the ring $A=\Gamma_0$.
\begin{prop}\label{prop:gamma-gen-by-degree-1}
For all $r\geq 1$, the product map $\mu\colon \Gamma_1\otimes_A
\Gamma_{r-1}\ra \Gamma_r$ is surjective.  In particular, $\Gamma$ is
generated as a ring by $\Gamma_0$ and $\Gamma_1$.
\end{prop}
\begin{proof}
Immediate using \eqref{lemma:socle-lemma}.
\end{proof}

\subsection{The complex $K^r_\bullet$ is a bar resolution of $\Gamma$}

We thus have the following description of the complex $K^r_\bullet$.
\begin{prop}
For $r\geq1$, there are isomorphisms of $A$-modules
\[
K_q^r \approx \bigoplus_{r_1+\cdots+r_q=r} \Gamma_{r_1}\otimes_A
\cdots\otimes_A \Gamma_{r_q},
\]
where the sum is taken over tuples $(r_1,\dots,r_q)$ of positive
integers which sum to $r$.  With respect to these isomorphisms, the
boundary map $\partial \colon K_q^r\ra K_{q-1}^r$ is given by
\[
\partial (\phi_1\otimes\cdots \otimes \phi_r) = 
\sum_{i=1}^{q-1}(-1)^i
\phi_1\otimes\cdots \otimes \phi_i\phi_{i+1}\otimes \cdots \phi_q,
\]
where $\phi_1\otimes\cdots \otimes \phi_r\in \Gamma_{r_1}\otimes_A
\cdots \otimes_A \Gamma_{r_q}\subseteq K^r_q$, and $\phi_1\otimes\cdots \otimes
\phi_i\phi_{i+1}\otimes_A \cdots \otimes_A \phi_q\in
\Gamma_{r_1}\otimes_A \cdots \otimes_A \Gamma_{r_i+r_{i+1}}\otimes_A
\cdots \otimes_A \Gamma_{r_q}\subseteq K^r_{q-1}$.
\end{prop}
\begin{proof}
  Immediate.
\end{proof}

This amounts to saying that the complex $K_\bullet = \bigoplus_r
K_\bullet^r$ is 
isomorphic to the normalized bar complex
$\overline{\mathcal{B}}(A,\Gamma,A)$ of the augmented associative
ring $\Gamma$.

\subsection{Relations in $\Gamma$}
\label{subsec:relations-in-gamma}

We now describe certain elements $P_i$ in $\Gamma$, and certain
relations among them; below it will be shown that this gives a
presentation of $\Gamma$ in terms of generators and relations.

For $0\leq i\leq p$, let $P_i\in \Gamma_1$ denote the element defined
by
\[
P_i(x_1^j) = 0\quad\text{if $i\neq j$}, \qquad P_i(x_1^i)=1,
\]
where $0\leq j\leq p$.
That is, $P_0,\dots,P_p$ is a left $A$-module basis of $\Gamma_1$,
dual to the monomial left $A$-module basis $1,x_1,\dots,x_1^p$ of
$A_1=k\powser{x_0,x_1}/(F_p(x_0,x_1))$.  

The right $A$-module structure on $\Gamma_1$ may be described as
follows.  A straightforward calculation shows that
for $c\in k$,
\begin{equation}\label{eq:gamma-formula-1}
P_i c = c^p P_i,
\end{equation}
and
for the generator $x\in k\powser{x}=A$, we have
\begin{equation}\label{eq:gamma-formula-2}
\begin{aligned}
    P_0 x &= -x^{p+1}P_p,
\\
    P_1 x &= P_0+xP_p,
\\
    P_i x &= P_{i-1} \qquad \text{(if $1<i<p$),}
\\
    P_p x &= P_{p-1}+x^pP_p.
  \end{aligned}
\end{equation}
(This amounts to the identity $x_1^{p+1}=-x_0^{p+1}+x_0x_1+x_0^px_1^p$ in
$A_1$.)

Observe that these imply that for each $i=0,\dots,p$, we have $P_ix^{p+1}=xQ_i$ for some element
$Q_i\in \Gamma_1$.  Thus, the above identities determine the structure of
$\Gamma_1$ as a right $A$-module; for, if $f(x)\in k\powser{x}$ is a
limit of a sequence of polynomials $f_n(x)$, we see that $P_if(x)$ is
the limit  as $n\to\infty$ of the sequence
$\{P_if_n(x)\}$ with respect to the $x$-adic topology.

The exact sequence of \eqref{lemma:relations} gives rise, on taking
duals, to a short exact sequence
\[
0\ra (A_1/s(A))^* \xra{\bar{v}^*} \Gamma_{1}\otimes_A \Gamma_1
\xra{\mu} \Gamma_2\ra 0. 
\]
The natural inclusion $(A_1/s(A))^*\subset A_1^*\approx \Gamma_1$
identifies $(A_1/s(A))^*$ with the left sub-$A$-module of $\Gamma_1$
spanned by $P_1,\dots,P_p$, and a straightforward calculation shows
that $\bar{v}^*(P_i) = \sum_{j=0}^p x^j\,P_i\otimes P_j\in
\Gamma_1\otimes_A \Gamma_1$.  That is,
the identity 
\begin{equation}\label{eq:gamma-formula-3}
P_iP_0+x\,P_iP_1+\cdots+x^p\,P_iP_p=0
\end{equation}
holds in the ring $\Gamma$ for each $i=1,\dots,p$.  (It does
\emph{not} hold for $i=0$.)

\subsection{The structure of $\Gamma$}
\label{subsec:structure-of-gamma}

Let $T\Gamma_1=\bigoplus_{r\geq0} \underbrace{\Gamma_1\otimes_A \cdots
  \otimes_A \Gamma_1}_{\text{$r$ factors}}$ denote the tensor algebra
on the $A$-bimodule $\Gamma_1$.  Let $\Delta=T\Gamma_1/J$, where $J$
is the two-sided ideal generated by $\sum_{j=0}^p x^j\,P_iP_j$ for
$i=1,\dots,p$.
Observe that since $J$ is generated by homogeneous elements, we have
$\Delta\approx \bigoplus_{r\geq0}\Delta_r$ where $\Delta_r$ is an
$A$-bimodule quotient of $\Gamma_1^{\otimes_A r}$.

A \dfn{sequence} will be a list $I=(i_1,\dots,i_r)$, of length $r\geq0$,
of elements of $\{0,1,\dots,p\}$.  We say such a sequence is
\dfn{inadmissible} if there exists a $k$ such that $i_k\neq0$ and
$i_{k+1}=0$; otherwise, it is \dfn{admissible}.  Thus, a sequence is
admissible precisely if all zeros in $I$ appear at the beginning of
the sequence.  

Given a sequence
$I=(i_1,\dots,i_r)$, 
we write $P_I=P_{i_1}\cdots P_{i_r}\in \Delta_r$.   

\begin{prop}  Let $r\geq1$.
\begin{enumerate}
\item [(i)] We have that
\[
\Delta_r = A\,P_0^r + \sum_{i=1}^p \Delta_{r-1}\,P_i.
\]
\item[(ii)]
As a left $A$-module $\Delta_r$ is spanned by the
elements $P_I$ where $I$ is admissible of length $r$.
\end{enumerate}
\end{prop}
\begin{proof}
We  prove (i) by induction on $r$; it is immediate for $r=1$.
Assuming $\Delta_r=A\,P_0^r +\sum_{i=1}^p \Delta_{r-1}P_i$, we have that
\begin{align*}
\Delta_{r+1} &= \Delta_r\,P_0+ \sum_{j=1}^p \Delta_r\, P_j
\\
&= A\,P_0^{r+1}+\sum_{i=1}^p \Delta_{r-1}\,P_iP_0 + \sum_{j=1}^p
\Delta_r\, P_j &\text{by induction,}
\\
&\subseteq A\,P_0^{r+1} + \sum_{i=1}^p \sum_{j=1}^p
\Delta_{r-1}(-x^i\,P_i)P_j + \sum_{j=1}^p\Delta_r\,P_j,
\end{align*}
which is contained in $A\,P_0^{r+1}+ \sum_{j=1}^p\Delta_r\,P_j$.

Statement (ii) follows from (i) and induction on $r$.
\end{proof}

Let $\zeta\colon \Delta\ra \Gamma$ denote the evident map of
graded associative rings, induced by sending $\Delta_1\subset \Delta$
identically 
to $\Gamma_1\subset \Gamma$; it exists according to the discussion of
\S\ref{subsec:relations-in-gamma}.   We write $\zeta_r\colon
\Delta_r\ra\Gamma_r$ for the restriction of $\zeta$ to the $r$th grading.
\begin{prop}
The map $\zeta\colon \Delta\ra \Gamma$ is an isomorphism.
\end{prop}
\begin{proof}
We show that $\zeta_r$ is an isomorphism, by induction on $r$.
It is clear that $\zeta_0$ and $\zeta_1$ are isomorphisms.  In the
commutative diagram
\[\xymatrix{
{\Delta_1\otimes_A \Delta_{r-1}} \ar[r]^{\zeta_1\otimes \zeta_{r-1}}
\ar[d]_{\mu_\Delta} & {\Gamma_1\otimes_A 
  \Gamma_{r-1}} \ar[d]^{\mu_\Gamma}
\\
{\Delta_r} \ar[r]_{\zeta_r} & {\Gamma_r}
}\]
the map $\mu_\Delta$ is surjective by construction, $\mu_\Gamma$ is
surjective by \eqref{prop:gamma-gen-by-degree-1}, and
$\zeta_1\otimes\zeta_{r-1}$ is an 
isomorphism by induction.  Therefore $\zeta_r$ is surjective.  We also
know that $\Gamma_r$ is free as a left $A$-module on
$1+p+\cdots+p^r$ generators, while $\Delta_r$ is generated as a
left $A$-module by
$1+p+\cdots+p^r$ elements (the admissible monomials of length
$r$).  Thus $\zeta_r$ is an isomorphism.
\end{proof}

\subsection{The monomial filtration on $K^r_\bullet$}

Given sequences $I$ and $J$, we write $IJ$ for their concatenation.
We define a linear ordering on the set of sequences of length $r$ as
follows.  We say that $I<J$ if, on writing $I=I'(i)$ and $J=J'(j)$, we
have either (1) $i>j$, or (2) $i=j$ and $I'<J'$.

Thus, sequences of length $1$ are ordered:
\[
(p)<(p-1)<\cdots<(1)<(0).
\]
Sequences of length $2$ are ordered:
\[
(p,p)<\cdots<(0,p)<(p,p-1)<\cdots<(1,1)<(0,1)<(p,0)<\cdots<(1,0)<(0,0).
\]

We write 
\[
\Gamma_{r_1,\dots,r_q} \defeq \Gamma_{r_1}\otimes_A \cdots \otimes_A
\Gamma_{r_q}.
\]
For a sequence $I$ of length $r=\sum_{i=1}^qr_i$, 
let $\cF_I\Gamma_{r_1,\dots,r_q}$ denote the left $A$-submodule of
$\Gamma_{r_1,\dots,r_q}$ spanned 
by elements of the form $P_{I_1}\otimes\cdots \otimes P_{I_q}$, where
$I_1I_2\cdots I_q\leq I$.  Thus we obtain a filtration with
$\cF_I\Gamma_{r_1,\dots,r_q}\subseteq \cF_J\Gamma_{r_1,\dots,r_q}$
when $I\leq J$.  Observe that $\cF_{(0,\dots,0)}\Gamma_{r_1,\dots,r_q}
= \Gamma_{r_1,\dots,r_q}$.

We write $\cF_{<I}\Gamma_{r_1,\dots,r_q}$ for the left $A$-submodule
spanned by elements $P_{I_1}\otimes\cdots\otimes P_{I_q}$ where
$I_1I_2\cdots I_q<I$.  We write $\gr_I\Gamma_{r_1,\dots,r_q}=
\cF_{I}\Gamma_{r_1,\dots,r_q}/\cF_{<I}\Gamma_{r_1,\dots,r_q}$.

\begin{prop}
Let $\mu_i\colon \Gamma_{r_1,\dots,r_q}\ra
\Gamma_{r_1,\dots,r_{i-1}+r_i,\dots,r_q}$ be the map induced by
multiplication $\Gamma_{r_{i-1}}\otimes_A \Gamma_{r_i}\ra
\Gamma_{r_{i-1}+r_i}$.  Then for any sequence $I$ of length
$r=r_1+\cdots+r_q$, we have that
\[
\mu_i(\cF_I\Gamma_{r_1,\dots,r_q})\subseteq
\cF_{I}\Gamma_{r_1,\dots,r_{i-1}+r_i,\dots, r_q},\qquad
\mu_i(\cF_{<I}\Gamma_{r_1,\dots,r_q})\subseteq
\cF_{<I}\Gamma_{r_1,\dots,r_{i-1}+r_i,\dots, r_q}.
\]
\end{prop}
\begin{proof}
The filtrations are defined as the left $A$-modules spanned by certain
monomial elements.  The map $\mu_i$ is left $A$-linear and preserves
the spanning sets.
\end{proof}

\textbf{Warning.}  This filtration does \textbf{not} make $\Gamma$ into a
filtered ring.  That is, we do not generally have $\cF_I\Gamma_r\cdot
\cF_{I'}\Gamma_{r'} \subseteq \cF_{II'}\Gamma_{r+r'}$, since the
subobjects $\cF_I\Gamma_r$ are not \emph{right} $A$-submodules.

Now we define 
\[
\cF_I K_q^r \defeq \bigoplus_{r_1+\cdots+r_q=r}
\cF_I\Gamma_{r_1,\dots,r_q} \subseteq K_q^r.
\]
The above proposition implies that $\cF_IK_\bullet^r$ is a subcomplex
of $K_\bullet^r$.  Note further that
\[
\gr_I K_q^r \defeq \cF_I K_q^r/\cF_{<I}K_q^r \approx
\bigoplus_{r_1+\cdots+r_q=r} \gr_I 
\Gamma_{r_1,\dots,r_q}.
\]

The next proposition shows that the associated gradeds $\gr_I K_q^r$
are free modules, with basis given by elements $P_{I_1}\otimes \cdots
\otimes P_{I_q}$ where $I=I_1\cdots I_q$ with each $I_1,\dots,I_q$
admissible. 
\begin{prop}\label{prop:ss-filtration-lemma}
Let $I_1,\dots,I_q$ be sequences of length $r_1,\dots,r_q$
respectively, and let $I=I_1\cdots I_q$.
\begin{enumerate}
\item If at least one of $I_1,\dots,I_q$ is inadmissible, then $\gr_I\Gamma_{r_1\,\dots,r_q}=0$.
\item If all $I_1,\dots,I_q$ are admissible, then
  $\gr_I\Gamma_{r_1,\dots,r_q}$ is a free left $A$-module on one
  generator corresponding to $P_{I_1}\otimes\cdots\otimes P_{I_q}$.
\end{enumerate}
\end{prop}
\begin{proof}
First note that by definition $\gr_I=\gr_I\Gamma_{r_1,\dots,r_q}$ is
always a cyclic 
$A$-module, generated by the image of
$P_{I_1}\otimes\cdots\otimes P_{I_q}$.

We prove (1).  Suppose that $I_k$ is inadmissible.  Then we may write
$I_k=I_k'(i,0)I_k''$, where $I_k'$ and $I_k''$ are two (possibly
empty) sequences, and $i\neq0$.  We have that
\begin{align*}
P_{I_1}\otimes\cdots \otimes P_{I_q} &= P_{I_1}\otimes\cdots \otimes
P_{I_k'}P_iP_0 P_{I_k''}\otimes\cdots\otimes P_{I_q}
\\
&= \sum_{j=1}^p P_{I_1}\otimes\cdots \otimes
P_{I_k'}(-x^j)P_iP_jP_{I_k''} \otimes \cdots \otimes P_{I_q}.
\end{align*}
For any $a\in A$, the element $P_{I_1}\otimes\cdots\otimes P_{I_k'}a$ is
in $\Gamma_{r_1,\dots,r_k'}$ (where $r_k'$ is the length of $I_k'$),
and so is a left $A$-linear combination of monomials of the form
$P_{J_1}\otimes\cdots \otimes P_{J_k}$.

Thus, $P_{I_1}\otimes\cdots\otimes P_{I_q}$ is a left $A$-linear
combination of monomials of the form $P_{J_1}\otimes\cdots
P_{J_k}P_iP_jP_{I_k''}\otimes\cdots \otimes P_{I_q}$ with $j\neq0$.
Since 
\[
J_1\cdots J_k(i,j)I_k''\cdots I_q< I_1\cdots
I_{k'}(i,0)I_{k''}\cdots I_q,
\] 
it follows that
$\gr_I\Gamma_{r_1,\dots,r_q}=0$, proving (1).

To prove (2), observe that from (1) we may conclude that
$\Gamma_{r_1,\dots,r_q}$ is spanned as a left $A$-module by elements
of the form $P_{I_1}\otimes \cdots \otimes P_{I_q}$ with
$I_1,\dots,I_q$ admissible.  Since $\Gamma_{r_1,\dots,r_q}$ is a free
left $A$-module, with rank equal to the number of such collections of
admissible sequences, the result follows.
\end{proof}

Given an abstract simplicial complex $X$ with some chosen ordering of
its vertices, let $C_\bullet(X)$ denote the
chain complex associated to $X$, and let $\widetilde{C}_\bullet(X)$ denote
the mapping fiber of the augmentation map $C_\bullet(X)\ra \Z$, where
$\Z$ is viewed as a chain complex concentrated in degree $0$.  Thus, 
$\widetilde{C}_q(X)$ is the free abelian group on the $q$-simplices of $X$
for $q\geq0$, and $\widetilde{C}_{-1}(X)=\Z$.

Let $\Delta^n$ denote the $n$-simplex viewed as a simplicial complex.
The vertices of $\Delta^n$ are elements of $S=\{1,\dots,n+1\}$, and a
$q$-simplex of $\Delta^n$ is a subset of size $q+1$ of $S$.  Observe
that $\Delta^{-1}$ is a simplicial complex whose realization is the
empty space.

The following is elementary and standard; it amounts to the fact that
the quotient $\len{\Delta^n}/\len{Y}$ of the $n$-simplex by a
subcomplex which is a union of codimension $1$ faces is either
contractible or  homeomorphic to a sphere.
\begin{prop}
If $\Delta^n$ is the $n$-simplex viewed as a simplicial complex, and
if $Y_1,\dots,Y_d\subset \Delta^n$ is a 
(possibly empty) collection of distinct codimension $1$ faces of $\Delta^n$,
then
\[
  H_q \left[\widetilde{C}_\bullet(\Delta^n)/\sum_{i=1}^d
    \widetilde{C}_\bullet(Y_i)\right] = 0\qquad 
  \text{if $q\neq n$ or $d<n+1$.}
\]
If $q=n$ and $d=n+1$, then $H_n(\widetilde{C}_\bullet(\Delta^n)/\sum
\widetilde{C}_\bullet(Y_i))=\Z$. 
\end{prop}

\begin{prop}
Let $I=(i_1,\dots,i_r)$ be a sequence.  Then there is an isomorphism
of chain complexes
\[
\gr_I K^r_\bullet \approx
\left[\widetilde{C}_{\bullet-2}(\Delta^{r-2})/
  \sum_{i=1}^d\widetilde{C}_{\bullet-2}(Y_i)\right]\otimes_\Z    A,
\]
where $Y_1,\dots,Y_d$ is a collection of distinct codimension $1$ faces of
$\Delta^{r-2}$.  Here $d$ is size of the set 
\[
T=\set{k\in\{1,\dots,r-1\}}{\text{$i_k\neq 0$ and $i_{k+1}=0$}}.
\]
Thus, $H_q\gr_I K_\bullet^r=0$ unless $q=r$, and $H_r\gr_I
K_\bullet^r$ is a free $A$-module (of rank $0$ or $1$, depending on $I$).
\end{prop}
\begin{proof}
Given a sequence $I=(i_1,\dots,i_r)$, we define maps 
\[
\phi_I\colon \widetilde{C}_{q-2}(\Delta^{r-2})\ra K_q^r= \bigoplus_{r_1+\cdots+r_q=r}
\Gamma_{r_1,\dots,r_q}
\]
as follows.  If  $s=[1\leq s_1<\cdots<s_{q-1}\leq r-1]$ is a $q-2$-simplex
in 
$\Delta^{r-2}$, then let $\phi_I(s) = P_{I_1}\otimes \cdots \otimes
P_{I_q}$, where 
\[
I_k = (i_{s_{k-1}+1},\dots,i_{s_k})\qquad\text{for $k=1,\dots,q$,
  taking $s_0=0$ and $s_q=r$.}
\]
Note that $I=I_1\cdots I_q$.

Thus $\phi_I\colon
\widetilde{C}_{\bullet-1}(\Delta^{r-2})\ra K_\bullet^r$ is a chain map,
and in fact the image of $\phi_I$ is contained in
$\cF_I K_\bullet^r$.   

Given $k\in \{1,\dots,r-1\}$, let $Y_k\subset \Delta^{r-2}$ denote
the subcomplex consisting of 
the codimension $1$ face spanned by all vertices except $k$.  Let
$Y_{k_1},\dots,Y_{k_d}$ be the collection of all such faces for which
$i_{k_j}\neq 0$ and
$i_{k_j+1}=0$.  By \eqref{prop:ss-filtration-lemma} (1) it follows that
$\phi_I$ factors through a map 
\[
\widetilde{C}_{\bullet-2}(\Delta^{r-2})/\sum_{j=1}^d
\widetilde{C}_{\bullet-2}(Y_{k_j}) \ra 
\cF_IK_\bullet^r/ \cF_{<I}K_\bullet^r
\]
By \eqref{prop:ss-filtration-lemma} (2), we see that this passes to an
isomorphism $A\otimes_\Z \widetilde{C}_{\bullet-2}(\Delta^{r-2})/\sum
\widetilde{C}_{\bullet-2}(Y_{k_j})\approx 
\gr_IK_\bullet^r$.
\end{proof}

\begin{prop}\label{prop:koszul-for-ss}
For all $r\geq0$, we have that $H_jK_\bullet^r=0$ if $j\neq r$, and
that $H_rK_\bullet^r$ is a finitely generated free $A$-module.
\end{prop}
\begin{proof}
The spectral sequence $E_1^{q,I}=H_q\gr_I K_\bullet^r \Longrightarrow
H_*K_\bullet^r$ collapses trivially, since $E_1^{q,I}=0$ unless $q=r$.
\end{proof}

\section*{Appendix: The polynomials $F_m(x,y)$}

For $m\in \N$ let $F_m(x,y)\in \Z[x,y]$ denote the
polynomial
\[
F_m(x,y) = \prod_{m=de} (x^d-y^e),
\]
where $d,e$ range over all pairs of natural numbers such that $m=de$.

Let $R$ be a commutative ring.  We propose to define a category $D(R)$
as follows.  
The objects of $D(R)$ are the elements of the ring $R$.
The morphisms are given by
\[
\Hom_{D(R)}(a,b)=\set{m\in \N}{F_m(a,b)=0}.
\]
We write $\langle m\rangle\colon a\ra b$ for the morphism corresponding to $m\in
\N$.  
Identity morphisms are those of the form $\langle 1\rangle\colon a\ra a$.

We define the composition of $\langle m\rangle \colon a\ra b$ with
$\langle n\rangle\colon b\ra
c$ to be $\langle mn\rangle\colon a\ra c$.  It is clear that this will
make $D(R)$ 
into a category, as long as composition is well-defined.  That is,
$D(R)$ is a category if $F_m(a,b)=0$ and $F_n(b,c)=0$ implies
$F_{mn}(a,c)=0$ for all $a,b,c\in R$ and $m,n\in \N$.  

We will show that with this composition law, $D(R)$ is in fact a
category for every $R$.  Equivalently, we show that for all $m,n$,
the polynomial $F_{mn}(x,z)$ 
is contained in the ideal $(F_m(x,y),F_n(y,z))$ of $\Z[x,y,z]$.

\begin{lemma}
If $R$ is an integral domain, then $D(R)$ is a category.  Thus, for
every $m,n$, there exists an $N\geq 1$ such that $F_{mn}(x,z)^N\in
(F_m(x,y), F_n(y,z))$.
\end{lemma}
\begin{proof}
If $a,b,c\in R$ satisfy $F_m(a,b)=0=F_n(b,c)$, then since $R$ is a
domain there must exist
$d,e,d',e'\in \N$ with $m=de$, $m'=d'e'$, such that $a^d=b^e$ and
$b^{d'}=c^{e'}$, whence $a^{dd'}=b^{d'e}=c^{ee'}$, whence
$F_{mn}(a,c)=0$.
\end{proof}

Let $T_m(x,y)=\Z[x,y]/(F_m(x,y))$.  

\begin{lemma}\label{lemma:F_n-usually-has-distinct-roots}
Let $K$ be a field of characteristic $0$, and let $\phi\colon \Z[x]\ra
K$  be a ring homomorphism such that $a=\phi(x)$ is neither $0$ nor a
root of unity.  Then $A=K\otimes_{\Z[x]}T_m(x,y)$ is isomorphic to a
finite product $\prod K_i$ of fields.  Furthermore, for each $i$ the
the evident
homomorphism $T_m(x,y)\ra A\ra K_i$ sends $y$ to an element
$b_i\in K_i$ which is neither $0$ nor a root of unity.
\end{lemma}
\begin{proof}
  We have that $A\approx K[y]/(T_m(a,y))$.  To show that $A$ is a
  product of fields, it suffices to show that
  the polynomial
  $T_m(a,y)\in K[y]$ has no repeated  roots in the algebraic closure $\bar{K}$
  of $K$.  The polynomial $T_m(a,y)$ is a product (up to sign) of
  factors of the form
   $g_d(y)=y^d-a^e$ where $m=de$.  It is clear that each $g_d$ has $d$
   distinct roots, of the form $\zeta \sqrt[d]{a^e}$ where $\zeta\in
   \mu_d(\bar{K})$ and $\sqrt[d]{a^e}$ some chosen $d$th root of
   $a^e$.  If $\beta\in \bar{K}$ such that
   $g_d(\beta)=0=g_{d'}(\beta)$ where $m=de=d'e'$ with $e>e'$, it is
   straightforward to show that $a^{e^2}=\beta^m = a^{{e'}^2}$, whence
   $a^{{e'}^2}(a^{e^2-{e'}^2}-1)=0$, which is impossible by the
   hypothesis on $a$.  Thus no roots of $T_m(a,y)$ are repeated.

   The homomorphism $T_m(x,y)\ra K_i$ sends $y$ to an element $b_i$ with
   the property that $b_i^{d}=a^e$ for some $m=de$.  Since $a$ is not
   $0$ or a root of unity, neither is $b_i$.
\end{proof}

\begin{prop}
  For all $m,n\geq1$, the polynomial $F_{mn}(x,z)$
  is an element of the ideal $(F_m(x,y),F_n(y,z))$ of $\Z[x,y,z]$.
  Thus, for every commutative ring $R$, $D(R)$ is a well-defined
  category. 
\end{prop}
\begin{proof}
It suffices to show that 
  the ring $T_m(x,y)\otimes_{\Z[y]}T_n(y,z) \approx
  \Z[x,y,z]/(F_m(x,y),F_n(y,z))$ has no nilpotents; since we have
  already shown that $F_{mn}(x,y)$ is nilpotent in this ring, we will
  thus have $F_{m,n}(x,y)\in (F_m(x,y),F_n(y,z))$. 

Let $K=\Q(x)$,
  viewed as a $\Z[x]$-algebra.  
The elements
  $F_m(x,y)$ are monic as polynomials in $y$ with coefficients in
  $\Z[x]$ (up to sign); thus the maps $T_m(x,y)\ra
  K\otimes_{\Z[x]}T_m(x,y)$ and $T_m(x,y)\otimes_{\Z[y]}
  T_n(y,z)\ra K\otimes_{\Z[x]}\otimes
  T_m(x,y)\otimes_{\Z[y]}T_n(y,z)$ are monomorphisms.  Hence, it
  suffices to show that $K\otimes_{\Z[x]}T_m(x,y)\otimes_{\Z[y]}
  T_n(y,z)$ has no nilpotents.  

By \eqref{lemma:F_n-usually-has-distinct-roots}, we see that
$K\otimes_{\Z[x]} T_m(x,y)\approx 
\prod_iK_i$ where $K_i$ are fields.  A second application of the lemma shows
that $K_i\otimes_{\Z[y]} T_n(y,z)\approx \prod_j K_{ij}$ where
$K_{ij}$ are fields, whence
$K\otimes_{\Z[x]}T_m(x,y)\otimes_{\Z[y]}T_n(y,z)\approx
\prod_{i,j}K_{ij}$, which clearly has no nilpotents.
\end{proof}

%%% bibliography

\end{document}